\documentclass{amsart}
\usepackage{hyperref}
\usepackage[utf8]{inputenc}
\usepackage{color,soul}
\usepackage{amsmath,amsthm,amscd,amssymb}
\usepackage{verbatim}
\theoremstyle{definition}

\newtheorem*{thm}{Theorem}

\newtheorem*{remark}{Remarks}

\numberwithin{equation}{section}
\def\RR{\mathbb{R}}

\def\FF{\mathbb{F}}
\def\ZZ{\mathbb{Z}}
\def\QQ{\mathbb{Q}}
\def\NN{\mathbb{N}}
\newcommand{\rbr}[1]{\left( {#1} \right)}

\newcommand{\cbr}[1]{\left\{ {#1} \right\}}

\usepackage[shortlabels]{enumitem}

\begin{document}

\author{Martin Grant, Kyle Hambrook, Alex Rusterholtz}
\title{L'H\^{o}pital's Rule is Equivalent to the Least Upper Bound Property}

\begin{abstract}
We prove that, in an arbitrary ordered field, L'H\^{o}pital's Rule is true if and only if the Least Upper Bound Property is true. We do the same for Taylor's Theorem 
with Peano Remainder, and for one other property sometimes given as a corollary of L'H\^{o}pital's Rule. 
\end{abstract}

\thanks{This material is based upon work supported by the National Science Foundation under
Award No. 2400329}

\maketitle

\section{Introduction}

In a first course in real analysis, the set of real numbers $\RR$ is often defined axiomatically 
as an ordered field satisfying the Least Upper Bound Property (also called the Dedekind Completeness Property). 
It is possible to define $\RR$ using a different property. 
For example, we could instead define $\RR$ as an ordered field satisfying the property that every bounded increasing sequence converges. 
We can do this because the property that every bounded increasing sequence converges is \textit{equivalent} to the Least Upper Bound Property -- 
two properties of an ordered field are called equivalent if every ordered field that satisfies one property satisfies the other, and conversely. 

Various other properties are known to be equivalent to the Least Upper Bound Property.
For example, the Intermediate Value Theorem, Extreme Value Theorem, and Mean Value Theorem are each equivalent to the Least Upper Bound Property. 
On the other hand, the Archimedean Property does not imply the Least Upper Bounded Property. 
Neither does the Cauchy Completeness Property. 
However, the Archimedean Property and the Cauchy Completeness Property, taken together, are equivalent to the Least Upper Bound Property.

The recent articles of Propp \cite{Propp}, Teismann \cite{Teismann}, and Deveau and Teismann \cite{DeveauTeismannBigList} 
show that a large number of statements are equivalent to the Least Upper Bound Property. 
Notably, L'H\^{o}pital's Rule was not discussed in \cite{Propp} or \cite{Teismann}. 
Deveau and Teismann \cite{DeveauTeismannBigList} showed that 
L'H\^{o}pital's Rule 
together with 
an additional property called 
Countable Cofinality (defined in the next section) 
imply the Least Upper Bound Property. 

In this paper, we prove that L'H\^{o}pital's Rule implies the Least Upper Bound Property without additional assumptions. 
Additionally, we show that Taylor's Theorem with Peano Remainder and the Limits of Derivatives Property (defined in the next section) each imply the Least Upper Bound Property.

\section{Definitions}

Let $\FF$ be an ordered field. Then $\FF$ contains canonically isomorphic copies of 
the set of rational numbers, the set of integers, and the set of positive integers. 
In a slight abuse of notation, we denote these copies by $\QQ$, $\ZZ$, and $\NN$, respectively. 
Also, $\FF$ is equipped with the order topology, i.e., the topology generated by the intervals of the form $(a,b)=\cbr{x \in \FF: a < x < b}$, where $a,b \in \FF$. 
The absolute value on $\FF$ is defined by $|x| = x$ if $x \geq 0$ and $|x| = -x$ if $x \leq 0$. 
For a function $f:\mathbb{F} \to \mathbb{F}$, we define limits, continuity, and derivatives in the usual way, except we replace $\RR$ by $\FF$. 
For example, for $a,L \in \FF$, the statement $\lim_{x \to a} f(x) = L$ means that for every positive $\epsilon \in \FF$ there is a positive $\delta \in \FF$ such that for every $x \in \FF$, if $0 < |x-a| < \delta$, then $|f(x)-L| < \epsilon$.

We now define some properties that $\FF$ \textbf{may} satisfy. All of these properties hold if $\FF = \RR$.

\begin{description}[leftmargin=0pt,topsep=0pt]
\setlength\itemsep{0.5em}
\item[\textbf{Least Upper Bound Property.}]   
Every subset of $\FF$ which is non-empty and bounded above has a least upper bound. 
(We say $\FF$ is complete if it satisfies the Least Upper Bound Property; otherwise, we say $\FF$ is incomplete.)  

\item[\textbf{Archimedian Property.}] 
For every $x \in \mathbb{F}$ there exists an $n \in \mathbb{N}$ such that $n > x$. 
(We say $\FF$ is Archimedean if it satisfies the Arch; otherwise, we say $\FF$ is non-Archimedean.)
\item[\textbf{Countable Cofinality.}]
There exists an unbounded sequence $(\lambda_n)_{n \in \NN}$ in $\FF$. 
\item[\textbf{Mean Value Theorem.}] 
If $f:\mathbb{F}\to\mathbb{F}$ is differentiable on $(a,b)$ and continuous on $[a,b]$, 
then there exists $c \in (a,b)$ such that $f(b)-f(a) = f'(c)(b-a)$. 
\item [\textbf{Taylor's Theorem with Lagrange Remainder.}] 
For every $n \in \NN$, 
if $f:\mathbb{F}\to\mathbb{F}$ is $n$ times differentiable on $(a,b)$ 
and $f^{(n-1)}$ exists and is continuous on $[a,b]$, 
then there exists $c \in \left( a,b \right)$ such that 
$$
f(b) - \sum_{k=0}^{n-1} \dfrac{f^{(k)}(a)}{k!} (b-a)^k  = \dfrac{f^{(n)}(c)}{n!}(b-a)^{n}. 
$$
\item [\textbf{Taylor's Theorem with Peano Remainder.}] 
For every $n \in \NN$, 
if $f:\mathbb{F}\to\mathbb{F}$ is $n$ times differentiable at $a$, then 
$$
\lim_{x \to a} \dfrac{1}{(x-a)^n} \rbr{ f(x) - \sum_{k=0}^{n} \frac{f^{(k)}(a)}{k!}(x-a)^k } = 0. 
$$
\item [\textbf{L'H\^{o}pital's Rule}.] 
If $f,g:\mathbb{F}\to \mathbb{F}$ are differentiable at $a$, $g'(a) \neq 0$, and $f(a)=f'(a)=g(a)=0,$  
then 
$$\lim_{x \to a}\frac{f(x)}{g(x)}=\frac{f'(a)}{g'(a)}.$$
\item [\textbf{Limit of Derivatives Property.}] 
If $f:\mathbb{F} \to \mathbb{F}$ is continuous at $a$, and $f'(x)$ exists for all $x$ in a punctured neighborhood of $a$, and $\lim_{x \to a} f'(x)$ exists, 
then $f$ is differentiable at $a$ and $\lim_{x \to a} f'(x) = f'(a).$ 
\end{description}

We emphasize again that these are properties an ordered field may or may not have. In the names of some of these properties we have used the word ``Theorem'' because that is what the property is in called in $\RR$. This does not mean the property is a true statement in an arbitrary ordered field. 

\begin{remark}\hspace{1pt} \vspace{-1em}
\begin{enumerate}[(1),leftmargin=2em,noitemsep,topsep=0pt]
\item Clearly, the Archimedean Property implies Countable Cofinality. But the converse is not true. For example, the ordered field of formal Laurent series is countably cofinal but not Archimedean; see, e.g., \cite{Teismann} for details. 

\item The Least Upper Bound Property, the Mean Value Theorem, and Taylor's Theorem with Lagrange Remainder are all equivalent. 
Indeed, most introductory textbooks on real analysis (see, e.g.,  \cite{Bartle}, \cite{RudinPMA}) 
prove that the Least Upper Bound Property implies the Mean Value Theorem, 
and that the Mean Value Theorem implies Taylor's Theorem with Lagrange Remainder. 
Moreover, the Mean Value theorem is simply the $n=1$ case of the Lagrange form of Taylor's Theorem. 
To see that the Mean Value Theorem implies the Least Upper Bound Property, 
suppose $A$ is a non-empty subset of $\FF$ which is bounded above but has no least upper bound.  
Take $a \in A$ and $b$ an upper bound for $A$. Let $f:\FF \to \FF$ to be the indicator function of $A \cup (-\infty,a)$. 
The set $A \cup (-\infty,a)$ and its complement are both open in $\FF$, so  
$f$ is continuous on $[a,b]$, and $f'(c) = 0$ for all $c \in (a,b)$. 
But $f(b) - f(a) = 0-1 \neq 0 = f'(c)(b-a)$.  

\item The $n=1$ case of Taylor's Theorem with Lagrange Remainder is exactly the Mean Value theorem, which is equivalent to the Least Upper Bound Property. 
On the other hand, the $n=1$ case of the Taylor's Theorem with Peano Remainder is true in every ordered field. 
(Indeed, the conclusion is essentially the definition of $f'(a)$.)  
These observations may motivate one to ask whether the full Taylor's Theorem with Peano Remainder is also true in an arbitrary ordered field. 
In fact, this was the question that motivated us to write this paper and our main theorem answers this question. 

\end{enumerate}
\end{remark}

\section{Main Result}

\begin{thm} 
In an arbitrary ordered field $\FF$, the following are equivalent. 
\begin{enumerate}[(a),topsep=0pt]
\item Least Upper Bound Property
\item L'H\^{o}pital's Rule
\item Taylor's Theorem with Peano Remainder 
\item Limit of Derivatives Property
\end{enumerate}
\end{thm}

\begin{proof}\hspace{1pt}
The most significant part of the proof is the implication (d) $\Rightarrow$ (a). 

(a) $\Rightarrow$ (b): A proof can be found in most real analysis textbooks (see e.g., \cite{Bartle}, \cite{RudinPMA}). 

(b) $\Rightarrow$ (c): Apply L'H\^{o}pital's Rule 
repeatedly to evaluate $\lim_{x \to a} F(x)/G(x)$,  
where $F(x) = f(x) - \sum_{k=0}^{n} \frac{f^{(k)}(a)}{k!}(x-a)^k$ and $G(x) = (x-a)^n$.

(c) $\Rightarrow$ (b): Apply Taylor's Theorem with Peano Remainder 
for $n=1$ 
to the functions $f(x)$ and $g(x)$, 
then using the results to evaluate $\lim_{x \to a} f(x)/g(x)$.

(b) $\Rightarrow$ (d): Apply L'H\^{o}pital's Rule 
%
%
to evaluate $f'(a) = \lim_{x \to a} F(x)/G(x)$, where $F(x) = f(x)-f(a)$ and $G(x)=x-a$. 

(d) $\Rightarrow$ (a): Assume $\FF$ is does not satisfy the least upper bound property.  We will show $\FF$ does not satisfy the Limit of Derivatives Property. 
The proof is split into two cases according to whether $\FF$ is Archimdean or non-Archimedean. 
The Archimdean case is inspired by Exercise 7.1.8 of \cite{Korner}. It also uses the standard fact that every Archimedean ordered field is isomorphic to a subfield of $\RR$. 
(This is proved by considering the map $\phi:\FF \to \RR$, $\phi(x) = \sup\cbr{q \in \QQ, q < x}$.) 
The non-Archmedean case uses the notion of Archimedean classes, which we learned about from \cite{Sierens}, though it seems to originate with \cite{Hahn:1907}.

\noindent Case 1: Assume $\FF$ is Archimedean. 
Then $\mathbb{F}$ is (isomorphic to) a proper subfield of $\RR$. 
Choose $c \in \RR \setminus \FF$. Using the Archimedean Property, choose $k \in \ZZ$ such that $k < c < k+1$. 
Define $c_0 = (c-k)(1 - 1/2) + 1/2$. Then $c_0 \in \RR \setminus \FF$ and $1/2 < c_0 < 1$. 
For each $n \in \ZZ$, define $c_n = 2^{-n} c_0$. 
Then $c_n \in \RR \setminus \FF$ and $2^{-(n+1)} < c_n < 2^{-n}$. 
Define $I_n = \cbr{x \in \FF : c_n < |x| < c_{n-1}} = (-c_{n-1}, -c_n) \cup (c_n, c_{n-1})$. 
Together with the set $\cbr{0}$, the sets $I_n$ form a partition of $\FF$. 
For each $n \in \ZZ_n$, fix an element $a_n \in I_n$. 
Define $f:\FF \to \FF$ by 
$$f(t):=\left\{
\begin{array}{ll}
      a_n & \text{if } t \in I_n \\
      0 & \text{if } t=0 
\end{array} 
\right.$$
Since $f$ is constant on each $I_n$, since each $I_n$ is open, and since each non-zero $t \in \FF$ belongs to some $I_n$, 
we have 
$$
f'(t) = \lim_{h \to 0} \dfrac{f(t+h)-f(t)}{h} = \lim_{h \to 0} \dfrac{0}{h} = 0 
$$
for each non-zero $t \in \FF$. 
Thus 
$$
\lim_{t \to 0} f'(t) = 0. 
$$
Moreover, for each non-zero $t \in \FF$, we have $t \in I_n$ for some $I_n$, so that  
$$
\dfrac{1}{2}|t| < \dfrac{1}{2}c_{n-1} = c_n < |f(t)| = |a_n| < c_{n-1} = 2c_n < 2|t|. 
$$
and hence 
$$
\dfrac{1}{2} < \dfrac{|f(t)|}{|t|} < 2. 
$$
It follows that 
$$
\lim_{t \to 0} f(t) = 0 = f(0)
$$ 
and 
$$
f'(0) = \lim_{t \to 0} \dfrac{f(t) - f(0)}{t - 0} = \lim_{t \to 0} \dfrac{f(t)}{t} \neq 0 = \lim_{t \to 0} f'(t). 
$$
Therefore $\FF$ does not satisfy the Limit of Derivatives Property.

\noindent Case 2: Assume $\FF$ is not Archimedean. 
For each $p,q \in \FF$, define:  
\begin{itemize}
\item $p \ll q$ if $n|p| < |q|$ for all $n \in \NN$. 
\item $p \sim q$ if there exist $m,n \in \NN$ such that $n|p| > |q|$ and $m|q| > |p|$. 
\end{itemize}
Note that $\sim$ is an equivalence relation on $\FF$. Thus its equivalence classes partition $\FF$. 
(The equivalence classes are known as the Archimedean classes of $\FF$.) 
The equivalence class of $0$ is $\cbr{0}$. 
Each non-zero equivalence class is open. Indeed, if $p$ is equivalent to a non-zero $q \in \FF$, then every element of $(p/2,2p)$ is also equivalent to $q$. 
From each non-zero equivalence class, choose a representative element. 
Let $\cbr{a_{n}}_{n \in Z}$ be the set of representatives ($Z$ is simply some index set).  
Define $f:\FF \to \FF$ by 
$$f(t):=\left\{
\begin{array}{ll}
      a_n & \text{if } t \sim a_n \\
      0 & \text{if } t \sim 0
\end{array} 
\right.$$
Since $f$ is constant on each non-zero equivalence class, since each non-zero equivalence class is open, and since each non-zero $t \in \FF$ belongs to some non-zero equivalence class, 
we have 
$$
f'(t) = \lim_{h \to 0} \dfrac{f(t+h)-f(t)}{h} = \lim_{h \to 0} \dfrac{0}{h} = 0 
$$
for each non-zero $t \in \FF$. 
Thus 
$$
\lim_{t \to 0} f'(t) = 0.
$$
Since $\FF$ is non-Archimedean, there exists $C \in \FF$ such that $0 < 1/C \ll 1 \ll C$. 
For each non-zero $t \in \FF$, we have $t \sim a_n$ for some $a_n$, so that  
$$
\dfrac{1}{C}|t| \ll |f(t)| = |a_n| \ll C|t|
$$
and hence 
$$
\dfrac{1}{C} < \dfrac{|f(t)|}{|t|} < C. 
$$
It follows that 
$$
\lim_{t \to 0} f(t) = 0 = f(0)
$$ 
and 
$$
f'(0) = \lim_{t \to 0} \dfrac{f(t) - f(0)}{t - 0} = \lim_{t \to 0} \dfrac{f(t)}{t} \neq 0 = \lim_{t \to 0} f'(t). 
$$
Therefore $\FF$ does not satisfy the Limit of Derivatives Property. 
\end{proof}

\bibliographystyle{plain}
\bibliography{Master}

\end{document}